\documentclass{amsart}
\usepackage{graphicx}

\theoremstyle{plain}
\newtheorem{thm}{Theorem}[section]
\newtheorem{cor}[thm]{Corollary}
\newtheorem{lem}[thm]{Lemma}
\newtheorem{prop}[thm]{Proposition}

\theoremstyle{definition}
\newtheorem{dfn}[thm]{Definition}

\theoremstyle{remark}
\newtheorem{rem}[thm]{Remark}
\numberwithin{equation}{section}

\begin{document}
\title[On MST-like Metric Spaces] % running head version
{On Minimum Spanning Tree-like Metric Spaces}
\author{Momoko Hayamizu$^{\ast\dagger}$}
\thanks{*~Department of Statistical Science, The Graduate University of Advanced Studies}
\author{Kenji Fukumizu$^{\dagger\ast}$}
\thanks{\dag~The Institute of Statistical Mathematics}
\curraddr[Momoko Hayamizu, Kenji Fukumizu]{Institute of Statistical Mathematics\\
 10-3 Midori-cho\\ Tachikawa, Tokyo 190-8562\\ Japan}
\email[]{\{hayamizu,fukumizu\}@ism.ac.jp}

\maketitle

\begin{abstract}
We attempt to shed new light on the notion of `tree-like' metric spaces by focusing on an approach that does not use the four-point condition. Our key question is: Given metric space $M$ on $n$ points, when does a fully labelled positive-weighted tree $T$ exist on the same $n$ vertices that precisely realises $M$ using its shortest path metric? 
We prove that if a spanning tree representation, $T$, of $M$ exists, then it is isomorphic to the unique minimum spanning tree in the weighted complete graph associated with $M$, 
and we introduce a \emph{fourth-point condition} that is necessary and sufficient to ensure the existence of $T$ whenever each distance in $M$ is unique. In other words, a finite median graph, in which each geodesic distance is distinct, is simply a tree. Provided that the tie-breaking assumption holds, the fourth-point condition serves as a criterion for measuring the goodness-of-fit of the minimum spanning tree to $M$, i.e., the spanning tree-likeness of $M$. It is also possible to evaluate the spanning path-likeness of $M$. These quantities can be measured in $O(n^4)$ and $O(n^3)$ time, respectively.
\end{abstract}
\section{Introduction}\label{sec:intro}
Historically, graphs as finite metric spaces have been extensively studied~\cite{Ha}. Even though we approach them differently, we would like to emphasise, amongst others~\cite{I,P1,P2,S}, the classical result provided by Buneman~\cite{B}. In short, a metric on a finite set can be realised by the shortest path metric in a positive-weighted tree if and only if it satisfies the four-point condition.
Not only is it frequently quoted in the context of evolutionary trees~\cite{SS}, but
it is also known for its direct connection to the theory of Gromov hyperbolic metric spaces~\cite{G}. Approximately two decades later after Buneman's theorem, Hendy~\cite{H} proved the existence of a unique tree representation for every metric satisfying the four-point condition.

Given this background, a metric space that satisfies the four-point condition is commonly considered tree-like. However, an important caveat should be addressed: 
the four-point condition is necessary and sufficient to ensure the existence of a \emph{partially labelled} tree that realises a given metric~\cite{Ha,H,SS}.
For example, a complete graph with a uniform edge length clearly satisfies the four-point condition, but it only becomes tree-like after an extra vertex is added. In this case, the four-point condition does not ensure that  a metric  is  realised by  a \emph{fully labelled} tree on \emph{the same} set. It does not characterise the distance within trees, in general, but rather the shortest path metrics induced by graphs of a certain class, called \emph{block graphs} (\textit{i.e.}, graphs in which all biconnected components  are complete subgraphs)~\cite{Ba}.

This may not create an issue in the field of conventional phylogenetics, but considering the recent surge of renewed biological interest in minimum spanning tree (MST)-based tree estimation~\cite{Q}, determining when a metric space is realised by a positive-weighted tree on the same set is not only a natural undertaking but also a meaningful one.
Thus far, this problem has not been properly recognised, much less addressed. The only two exceptions to this are the recent work provided in~\cite{A} and in~\cite{HEF}. It seems to be a non-trivial question not only because it cannot be answered using Buneman's theorem, but also because it is equivalent to determining a method for recognising a special case of the metric travelling salesman problem (TSP). If an input---a metric on a set of cities---is the shortest path metric in a tree on the city set,  
the length of the optimal tour must equal twice the length of the MST.

In this paper, we examine the sub-type of tree metrics without relying on the four-point condition. 
Our work is based on three ingredients: 
the  so-called tie-breaking assumption, which has been popular in algorithmic applications since the work provided by Kruskal in~\cite{K}; 
what we call  the fourth-point condition, which can typically be found in the definition of median metric spaces~\cite{DD}; and a simple trick for metric-preserving edge removal, which applies to any finite metric space.
These concepts, which are part of our original results, are defined and discussed in  Section~\ref{sec:preliminaries}.

As expected, if it exists, a fully labelled positive-weighted tree that realises a finite metric space is the unique MST in its associated weighted complete graph (Proposition~\ref{prop:1}).  Our goal is to prove the following: 
A finite metric space under the tie-breaking rule is realised by the MST if and only if it satisfies the fourth point condition (Theorem~\ref{thm:1}). This implies that every finite median graph, in which the shortest path lengths between all pairs of vertices are distinct, is necessarily a tree (Corollary~\ref{cor:2}). This result also yields a stronger condition for understanding when a finite metric space is realised, especially by a spanning path graph (Corollary~\ref{cor:1}).
We define and discuss the notion of a spanning tree-likeness of a finite metric space in Section~\ref{sec:discussion}.
%We provide the notions of a \emph{spanning tree-likeness} and a \emph{spanning path-likeness} of a finite metric space and contrast them with the Gromov hyperbolicity in  Section~\ref{sec:discussion}.

\section{Preliminaries}\label{sec:preliminaries}
We apply the metric-related terminology provided in \cite{DD} throughout this paper.
Let $(X,d_M)$ be a \emph{finite metric space}, that is, a finite set, $X$, equipped with metric $d_M$.  
For two distinct points $x$ and $x^{\prime}$ in $X$, the \emph{closed metric interval} between them is defined to be the set 
\[
I(x,x^{\prime}):=\{i\in X: d_M(x,x^{\prime})=d_M(x,i)+d_M(i,x^{\prime})\}.
\]

All graphs considered in this paper will be  simple, undirected,  
\emph{fully labelled} (\textit{i.e.}, each vertex is labelled), and \emph{positive weighted} (\textit{i.e.}, each edge has a positive length).
A graph is denoted $(V,E;w)$ for a set, $V$,  of labelled vertices and a set, $E$,  of edges that are associated with  a positive edge-weighting function, $w: E\mapsto \mathbb{R}^{+}$.
Given graph $G$, the sets of vertices and edges are denoted $V(G)$ and $E(G)$,  respectively. Moreover, graph $G$ is said to be a graph \emph{on} $V(G)$.
Vertices may be renamed as needed, assuming no confusion arises, and a vertex labelled `$x$' is referred to as vertex $x$. 
The distance in graph $G$ is defined to be the shortest path metric and is represented using $d_G$.

Assume $M$ is a finite metric space, $(X,d_M)$. 
Let $K_M$ be the associated weighted complete graph with $M$.
An edge of $K_M$ that joins two distinct vertices, $x$ and $x^{\prime}$, is denoted $e(x,x^{\prime})$. This paper uses the terms `\emph{points}' and `\emph{vertices}' interchangeably because there is a one-to-one correspondence between $X$ and $V(K_M)$ for any finite metric space $M$.

\subsection{Tie-breaking rule}\label{subsec:tiebreak}
\begin{dfn}\label{dfn:tiebreak}
A finite metric space, $(X,d_M)$, is said to satisfy \emph{the tie-breaking rule} if the values of $d_M$ are distinct for all pairs in $X$.
\end{dfn}

\subsection{The fourth point condition}\label{subsec:4thpc}
\begin{dfn}[Figure~\ref{fig:satisfy4thpc}]\label{dfn:4thPC}
A finite metric space, $(X,d_M)$, is said to satisfy \emph{the
fourth-point condition} if, for every (not necessarily distinct) three
points $x,y,z\in X$, there exists a point, $p^{*}\in X$, such that
\[
d_M(x,p^{*})+d_M(y,p^{*})+d_M(z,p^{*})=\frac{1}{2} \{d_M(x,y)+d_M(y,z)+d_M(z,x)\}.
\]
\end{dfn}

\begin{figure}[htbp]
\centering
\includegraphics[width=0.18\textwidth]{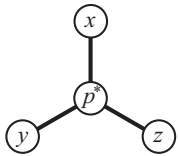}
\caption{Fourth point $p^*$ for triplet $\{x,y,z\}$ \label{fig:satisfy4thpc}}
\end{figure}

\begin{prop}\label{prop:4th_unique}
If a finite metric space, $(X,d_M)$, satisfies the fourth-point condition, 
 fourth point $p^*\in X$ is unique for each triplet in $X$. 
\end{prop}
\begin{proof}
Suppose that there are two quartets, $\{x,y,z,p_1^*\}$ and $\{x,y,z,p_2^*\}$ ($p_1^*\neq p_2^*$), in $X$ such that 
\[d_M(x,p_1^{*})+d_M(y,p_1^{*})+d_M(z,p_1^{*})= d_M(x,p_2^{*})+d_M(y,p_2^{*})+d_M(z,p_2^{*}).\]
Because $d_M$ is a metric on $X$, we have 
\begin{align*}
d_M(x,p_2^{*})&\leq d_M(x,p_1^{*})+d_M(p_1^{*},p_2^{*});\\
d_M(y,p_2^{*})&\leq d_M(y,p_1^{*})+d_M(p_1^{*},p_2^{*});\\
d_M(z,p_2^{*})&\leq d_M(z,p_1^{*})+d_M(p_1^{*},p_2^{*}).
\end{align*}
Therefore, $p_1^{*}=p_2^{*}$, but this is a contradiction. 
Hence, if $p^*$ exists for $\{x, y, z\}$, it is unique.
\end{proof}

\begin{prop}\label{lem:4thPC_alternative}
The following is equivalent to saying that finite metric space $(X,d_M)$ satisfies the fourth-point condition:  
For every (not necessarily distinct) three points $x, y, z \in X$, there
exists only one point  $p^*\in I(x,y)\cap I(y,z)\cap I(z,x)$.
\end{prop}

\begin{proof}
Because $d_M$ is a metric, for all $x,y,z,p\in X$, we have 
$d_M(x,p)+d_M(y,p)+d_M(z,p)\geq \frac{1}{2}\{d_M(x,y)+d_M(y,z)+d_M(z,x)\}$. 
The equality holds if and only if
$I(x,y)\cap I(y,z)\cap I(z,x)=\{p^*\}$. 
Proposition \ref{prop:4th_unique} ensures the uniqueness of $p^*$.
\end{proof}

\begin{rem}\label{rem:median}
Fourth point $p^*$ is also known as \emph{the median} for $\{x,y,z\}$  
because it minimises the sum of the distances to the three points, and 
a metric space satisfying the fourth-point condition (or a graph inducing this kind of metric space) is said to be \emph{median}~\cite{Ba,DD}.
Although a discussion of this topic is provided in~\cite{Ba,BCE}, it should be noted that median graphs include multiple types of graphs other than trees, such as grid and square graphs.
\end{rem}

\begin{lem}\label{lem:cycle}
Let $C$ be a cycle graph, $(V,E;w)$, with $\sum_{e\in E} w(e) =c$. 
Also, let $d_C$ be the shortest path metric in $C$.
Given three distinct points  $x,y,z\in V$  such that 
$d_C(x,y)+d_C(y,z)+d_C(z,x)=c$, 
the fourth point, $p^*$, exists in $V$  
if and only if $\max\{d_C(x,y),d_C(y,z),d_C(z,x)\}=c/2$.
\end{lem}
\begin{proof}
Without loss of generality, we can assume $d_C(z,x)=\max\{d_C(x,y), d_C(y,z), d_C(z,x)\}$. Clearly, $y\in I(x,y)\cap I(y,z)$. Therefore, $I(x,y)\cap I(y,z)\cap I(z,x)=\emptyset$ if and only if $y\not\in I(z,x)$. 
Under the assumption that the length of $C$ is fixed at $c$, 
this is equivalent to stating that $d_C(z,x)\neq c/2$. 
Thus, $I(x,y)\cap I(y,z)\cap I(z,x)= \emptyset$ if and only if $d_C(z,x)\neq c/2$. Applying Proposition~\ref{lem:4thPC_alternative} completes the proof.
\end{proof}

\subsection{Basic geodesic graphs}\label{subsec:basic}
In this subsection, we present a simple trick for metric-preserving edge removal, which can be used to represent an arbitrary finite  metric space  as a graph with the fewest edges. 
Let $M$ be a finite metric space, $(X,d_M)$, 
and assume $K_M$ is the weighted complete graph associated with $M$.

\begin{dfn}\label{dfn:realise}
Suppose $G$ is a connected graph on finite set $X$ with shortest path metric $d_G$.
Graph $G$ is said to \emph{realise $M$} if
$d_G(x,x^{\prime}) = d_M(x,x^{\prime})$  
for all  $x,x^{\prime}\in X$.
\end{dfn}

\begin{dfn}\label{dfn:circularperm}
Given $x,x^{\prime}\in X$, 
the edge, $e(x,x^{\prime})$, of $K_M$ is said to be \emph{non-basic} 
if there is a  permutation, $(x_1, \ x_2, \ \cdots, \ x_k)$, on a non-empty subset of $X\setminus\{x,x^{\prime}\}$ such that cyclic permutation $(x, \ x_1, \ x_2, \ \cdots , \ x_k, \ x^{\prime})$  satisfies 
\[
d_M(x,x^{\prime})=d_M(x,x_1)+d_M(x_1,x_2)+...+d_M(x_k,\,x^{\prime}). 
\]
The edge is called \emph{basic} otherwise.
\end{dfn}

\begin{prop}\label{rem:4thpoint_notexist}
Let $x,y,z$ be  three different vertices of $K_M$.
When the three edges, $e(x,y)$, $e(y,z)$, and $e(z,x)$, of $K_M$ are basic, the fourth point, $p^*$,   does not exist for $\{x,y,z\}$. 
If a non-basic edge exists, say $e(x,y)$,   
points  $x$ and $y$ are the only two candidates for $p^*$.
\end{prop}
The proof of this proposition is straightforward.

\begin{dfn}\label{dfn:basicgeodesicgraph}
Assume $B_M$ is the set of all basic edges of $K_M$, 
and suppose $\lambda$ is  a restriction of $d_M$ to $B_M$. 
A subgraph, $G_M:=(X, B_M; \lambda)$, in $K_M$ is called  
\emph{the basic geodesic graph in $K_M$}. 
\end{dfn}

\begin{lem}\label{lem:connected}
The basic geodesic graph, $G_M$, in $K_M$ is a connected graph on $X$ that realises $M$.
\end{lem}

\begin{proof}
It suffices to prove that $G_M$ is connected. 
Assuming that $e(x,x^{\prime})$ is non-basic, we show that there is a path of basic edges joining $x$ and $x^{\prime}$ in $K_M$. We also note that they are obviously connected in $G_M$ if $e(x,x^{\prime})\in E(K_M)$ is basic. 
Let $C$ be a cycle with the greatest number  of vertices (or edges) of all cycles in $K_M$ that share edge $e(x,x^{\prime})$ and overall length $2d_M(x,x^{\prime})$. Let  $V(C)=\{x,x^{\prime}\}\cup Y$, where   $Y:=\{x_1,\cdots,x_k\}$ is a non-empty subset of $X\setminus\{x,x^{\prime}\}$, as in  Definition~\ref{dfn:circularperm}. Furthermore, suppose $d_C$ is the shortest path metric induced by $C$ and $x_i, x_j\in V(C)$. 
If a path existed in $K_M$ joining $x_i$ and $x_j$ that was shorter than  $d_C(x_i,x_j)$, then edge $e(x,x^{\prime})$ would be longer than the path connecting  $x$ and $x^{\prime}$ through  $x_i$ and $x_j$. Therefore, any path in $K_M$ joining two vertices in $V(C)$  must have a length  greater than or equal to $d_C(x_i,x_j)$.  We use this fact  at the end of the proof.

In order to obtain a contradiction, we suppose $e(y,y^{\prime})\in E(C)\setminus e(x,x^{\prime})$  is  non-basic. We define   $C^{\prime}$  to be a cycle in $K_M$ of overall length $2d_M(y,y^{\prime})$  with   $e(y,y^{\prime})\in E(C^{\prime})$, which is similar to our previous case except that $|V(C^{\prime})|$ is unimportant. Let $V(C^{\prime})=\{y,y^{\prime}\}\cup Z$, where $Z:=\{y_1,\cdots,y_l\}\subseteq X\setminus\{y, y^{\prime}\}$.  By Definition \ref{dfn:circularperm}, if a cycle contains a non-basic edge, then it must be strictly longer than the other edges in the cycle. This implies that the number of non-basic edges contained in each cycle is zero or one. Thus, $e(y,y^{\prime})$ is shorter than $e(x,x^{\prime})$, and $e(y,y^{\prime})$ is the longest edge in $E(C^{\prime})$. Therefore, we can conclude that $e(x,x^{\prime})$ is not in $E(C^{\prime})$.
The assumption on $|V(C)|$ provides $Y\cap Z\neq\emptyset$. Our hypothesis ensures a path in $K_M$ of length $d_C(y,y^{\prime})$ that connects  $y$ and $y^{\prime}$ via $y^{\prime\prime}\in Y\cap Z$. This implies that  $K_M$ contains a path joining  $y$ and $y^{\prime\prime}$  of length less than $d_C(y,y^{\prime})$. If we  assume that  $y^{\prime}$ lies in the shortest path joining $y$ and $y^{\prime\prime}$ in $C$ (note that the roles of $y$ and $y^{\prime}$ can be exchanged), then we have  $d_C(y,y^{\prime})<d_C(y,y^{\prime\prime})$. It follows that there is a path that joins $y$ and $y^{\prime\prime}$ in $K_M$ of length less than $d_C(y,y^{\prime\prime})$. This is a contradiction. Hence, $e(y,y^{\prime})$ is  basic, which completes the proof.
\end{proof}

\begin{dfn}\label{dfn:treepath}
Finite metric space $M$ is said to be a \emph{spanning tree metric space} if the basic geodesic graph, $G_M$,  in the weighted complete graph, $K_M$, is a spanning subtree in $K_M$.  In particular, $M$ is said to be a \emph{spanning path metric space} if $G_M$ is a \emph{path graph} (\textit{i.e.}, a tree with two vertices of degree one and remaining vertices of degree two) that spans all the vertices of $K_M$. 
\end{dfn}

\begin{prop}\label{prop:1}
Let $M:=(X,d_M)$ be a spanning tree metric space and 
$G_M$ be the basic geodesic graph in  $K_M$. Then the following statements hold:
\begin{enumerate}
\item $G_M$ is the unique minimum spanning tree in $K_M$;
\item $G_M$ is the unique fully labelled tree  on $X$ that realises $M$.
\end{enumerate}
\end{prop}
\begin{proof}
(1) Assume $B:=E(G_M)$ and let $\overline{B}:=E(K_M)\setminus B$. Because $|B|=|X|-1$,  $G_M$ is the only spanning tree in $K_M$ such that all edges are basic.
In addition, let $e(x,x^{\prime})\in \overline{B}$.  Because $G_M$ is a tree, there is a  unique path  joining $x$ and $x^{\prime}$, denoted $P$. Each edge of $P$ must be  strictly shorter than $d_M(x,x^{\prime})$ for the following reasons: the length of $P$ equals $d_M(x,x^{\prime})$; the number of edges of $P$ exceeds one; and the edge weights are  all positive. Therefore,  replacing an arbitrary edge of $P$  with $e(x,x^{\prime})$ results in a spanning tree in $K_M$ of greater length. Hence, $G_M$ is shorter than any other spanning trees in $K_M$. 
(2) Suppose that $M$ is realised by fully labelled tree $T$ on $X$.  This implies that each edge of $T$ has a positive weight. 
We can recover $K_M$ from $T$ by summing the weights along every path in $T$ that has  two or more edges.
This process indicates that  $T$ is isomorphic to  the basic geodesic graph  in $K_M$. Hence, given (1), we know $T$ is unique.
\end{proof}
\begin{rem}
Proposition~\ref{prop:1} states that a metric space is uniquely realised by the only MST if it is a spanning tree metric space. Note that we do not need Buneman's four-point condition in  the argument ({\it cf.}~\cite{A}).
Concerning the uniqueness of the MST, the tie-breaking rule is a well-known sufficient condition established by Bor\r{u}vka~\cite{Bo} (cited in~\cite{K}) and by Kruskal~\cite{K}. The next section explores its relation to spanning tree metric spaces.
\end{rem}

\section{Main results}\label{sec:spanningtree}

\begin{thm}\label{thm:1}
Let $M$ be a finite metric space, $(X,d_M)$, under the tie-breaking rule.  
Then $M$  is a spanning tree metric space if and only if it satisfies the fourth-point condition. 
\end{thm}
\begin{proof}
(i)  The fourth-point condition clearly holds for all spanning tree metric spaces.
(ii) If $d_M$ is not a spanning tree  metric on $X$, then  
we will show that there is a triplet in $X$ that violates the fourth-point condition. 
According to Lemma \ref{lem:connected}, our assumption  implies that the basic geodesic graph, $G_M=(X,B;\lambda)$, in $K_M$ contains at least one cycle. 
Suppose $C:=(X_k,B_k; \lambda_k)$ is the shortest cycle  in $G_M$,  
where  $X_k\subseteq X$, $B_k\subseteq B$, $|X_k|=|B_k|=k$, and $\lambda_k$ is the restriction of $\lambda$ to $B_k$. Then Proposition \ref{rem:4thpoint_notexist} yields $k\geq 4$. Let $c$ denote the sum of the $\lambda_k$ over all elements in $B_k$. Also, assume that $d_C$ is the shortest path metric in $C$.
For all $i,j\in X_k$, no path in $G_M$ joining $i$ and $j$ has a shorter length than  $d_C(i,j)$  (otherwise, $C$ would not be the shortest cycle in $G_M$).
Therefore, $d_C(i,j)=\min \{a_{ij},c-a_{ij}\}$,  
in which $a_{ij}$  represents the length of the path in $C$ that travels from $i$ to $j$ in a clockwise direction. 

Consider a route in which we visit the points in $X_k$. Let $s\in X_k$ be the starting point from which we travel along the circle in a clockwise direction. 
We assign a label, `$\mathrm{L}$' or `$\mathrm{R}$', to every point $i\in X_k\setminus\{s\}$: label `$\mathrm{L}$' is assigned if $a_{si} < c/2$, and we use label  `$\mathrm{R}$'  if $a_{si}\geq c/2$.
If every point in $X_k\setminus\{s\}$ was labelled `$\mathrm{L}$', the last edge we would traverse returning to $s$ would be non-geodesic or non-basic. 
Therefore, there exists one and only one basic edge between vertices labelled `$\mathrm{L}$' and `$\mathrm{R}$'. Suppose that $t$ signifies the last point with label `$\mathrm{L}$' and $u$ indicates the first point with label `$\mathrm{R}$' as on the left  in Figure \ref{fig:thm1}. Note that  $d_M(s,t)+d_M(t,u)+d_M(u,s)=c$. 

We  assume that  $p^*$ exists for $\{s,t,u\}$ (otherwise, the assertion of the theorem immediately follows). Lemma~\ref{lem:cycle} gives us $\max\{d_M(s,t),d_M(t,u),d_M(u,s)\}=c/2$. Thus,  $d_M(u,s)=c/2$ (the edge joining $t$ and $u$ is basic, and $d_M(s,t)<c/2$). Let $v (\neq u)$ be a point  in $X_k$ with label `$\mathrm{R}$' that is between $u$ and $s$ as on the right in Figure \ref{fig:thm1}. We know point $v$ exists because $e(u,s)$ would be non-basic otherwise. According to the tie-breaking rule, we note that $a_{tv} \neq c-a_{tv}$. We can also set $a_{tv} < c-a_{tv}$ in order to select $\{s,t,v\}$. Although we should select $\{t,u,v\}$ when  $a_{tv} > c-a_{tv}$, we limit our consideration to the former case. Therefore, we have $d_M(s,t)+d_M(t,v)+d_M(v,s)=c$ again, but each of the three terms does not equal $c/2$ (recall that  $d_M(s,u)=c/2$). Hence, Lemma~\ref{lem:cycle} implies that  $p^*$ does not exist for $\{s,t,v\}$, and this completes the proof.
\end{proof}

\begin{figure}[htbp]
\centering
\includegraphics[width=0.4\textwidth]{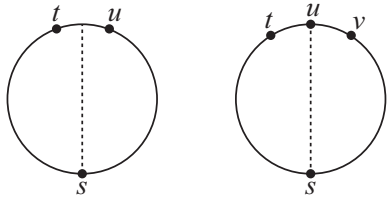}
\caption{Points in the proof of Theorem~\ref{thm:1} \label{fig:thm1}}
\end{figure}

\begin{rem}
Given a finite metric space on $X$, we can determine in $O(|X|^4)$ time whether it is a spanning tree metric space.
\end{rem}

\begin{cor}\label{cor:2}
Let $G$ be a median graph on finite set $X$ and  let $d_G$ be the shortest path metric of $G$. If each pair in $X$ has a different value for $d_G$, then $G$ is a tree.
\end{cor}

\begin{rem}\label{rem:insufficient}
As was mentioned in Remark~\ref{rem:median}, 
the fourth-point condition \textit{per se} is not a sufficient condition, but it is a  necessary condition in order to ensure that a finite metric space 
is induced by the shortest path metric in a tree 
(\textit{cf.}\ a  cycle graph on four vertices with a uniform edge length).
\end{rem}

\begin{cor}\label{cor:1}
Suppose $M:=(X,d_M)$ is a finite metric space  under the tie-breaking rule.  
Then $M$  is a spanning path metric space (Definition~\ref{dfn:treepath}) if and only if it satisfies \emph{the three-point condition}: for every (not necessarily distinct) three points $x,y,z\in X$, we have 
\[
\max\{d_M(x,y),d_M(y,z),d_M(z,x)\}=\frac{1}{2}\{d_M(x,y)+d_M(y,z)+d_M(z,x)\}. 
\]
The condition can be confirmed in $O(|X|^3)$ time. If $M$ is a spanning path metric space, it is realised by the unique shortest path that joins the farthest two points  in $X$. 
\end{cor}
\begin{proof}
We only prove the first statement. The three-point condition obviously holds for all spanning path metric spaces. Therefore, we assume that the three-point condition holds and show that the basic geodesic graph, $G_M$, in $K_M$ is a path graph on $X$. It is clear that $y$ is the fourth point, $p^{*}$, for $\{x,y,z\}$ when  the left-hand side equals $d_M(z,x)$. This means that the fourth-point condition automatically holds for any finite metric space that satisfies the three-point condition. Therefore, our assumption implies that $G_M$  is a tree on $X$. The three-point condition also indicates that every vertex in $G_M$ has a degree of one or two. In other words, if vertex $x$ has degree three or more, then any three distinct vertices adjacent to $x$ would violate the three-point condition. Hence, $G_M$ is a path graph on $X$, which completes the proof.
\end{proof}

\section{Discussion}\label{sec:discussion}
The hyperbolicity of finite metric spaces (or graphs) is a concept provided by Gromov~\cite{G,SS} and measures the deviance of a metric space from Buneman's four-point condition. If a metric space, $M$, satisfies the four-point condition, then the hyperbolicity of $M$ equals $0$, and $M$ is said to be $0$-hyperbolic. As was previously discussed, any complete graph with a uniform edge length is $0$-hyperbolic. Because the four-point condition is a stronger version of the triangular inequality, all metric triangles are also $0$-hyperbolic. Therefore, although the value of hyperbolicity is usually called the `tree-likeness' of $M$, a more precise interpretation refers to the \emph{partially labelled} tree-likeness of $M$. 
Therefore, as a final remark, we provide the notion of a \emph{fully labelled} tree-likeness of $M$.

Let us say that finite metric space $M$  is  \emph{$\rho$-roundabout}. Here,  $\rho$ is defined to be 
%\[
%\max_{x,y,z\in X}\min_{i\in X}\;\{d_M(x,i)+d_M(y,i)+d_M(z,i)-\frac{1}{2}(d_M(x,y)+d_M(y,z)+d_M(z,x))\}.
%\]
\[
\max_{x,y,z\in X}\min_{i\in X}\;\frac{d_M(x,i)+d_M(y,i)+d_M(z,i)}{d_M(x,y)+d_M(y,z)+d_M(z,x)}-\frac{1}{2}.
\]
This measures how far $M$ deviates from the fourth-point condition.   Provided that the tie-breaking rule holds, the value of $\rho$ can be regarded as the spanning tree-likeness of $M$ or the circuitousness of $d_M$ as illustrated in Figure~\ref{fig:roundabout}. Note that $\rho$ is invariant under multiplication of $d_M$ by a constant. As we have already seen, $M$ is $0$-roundabout if and only if there is an exact fit between $M$ and the MST. 

\begin{figure}[htbp]
\centering
\includegraphics[width=0.4\textwidth]{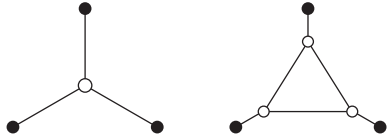}
\caption{Illustrations of spanning tree-likeness ($\rho=0$ and $\rho>0$)\label{fig:roundabout}}
\end{figure}

The degree of violation of the three-point condition similarly provides the spanning path-likeness of $M$---the maximum discrepancy between the left and right-hand sides of the triangular inequality. On the other hand, hyperbolicity does not provide any information because all metric triangles are $0$-hyperbolic.

\section*{Acknowledgement}
This work was supported in part by JSPS KAKENHI Grant Number 25120012 and 26280009. 
% The authors thank the anonymous reviewers for their careful reading and helpful comments. 
Special thanks are extended to Yoshimasa Uematsu for his insights concerning the concept of median.

\bibliographystyle{amsplain}
\bibliography{manuscriptMH}

\providecommand{\bysame}{\leavevmode\hbox to3em{\hrulefill}\thinspace}
\providecommand{\MR}{\relax\ifhmode\unskip\space\fi MR }
% \MRhref is called by the amsart/book/proc definition of \MR.
\providecommand{\MRhref}[2]{%
  \href{http://www.ams.org/mathscinet-getitem?mr=#1}{#2}
}
\providecommand{\href}[2]{#2}
\begin{thebibliography}{10}

\bibitem{A}
A.~Baldisserri, \emph{Buneman's theorem for trees with exactly $n$ vertices},
  arXiv:1407.0048v1 [math.CO] (2014), preprint.

\bibitem{Ba}
H-J. Bandelt and V.~Chepoi, \emph{Metric graph theory and geometry: a survey},
  Surveys on discrete and computational geometry, Contemp. Math., vol. 453,
  Amer. Math. Soc., Providence, RI, 2008, pp.~49--86. \MR{2405677
  (2009h:05068)}

\bibitem{BCE}
H.-J. Bandelt, V.~Chepoi, and D.~Eppstein, \emph{Combinatorics and geometry of
  finite and infinite squaregraphs}, SIAM J. Discrete Math. \textbf{24} (2010),
  no.~4, 1399--1440. \MR{2735930 (2012f:05072)}

\bibitem{Bo}
O.~Bor{\r{u}}vka, \emph{O jist{\'{e}}m problem minim{\'{a}}ln{\'{i}}m (about a
  certain minimal problem)}, Pr{\'{a}}ce morav. p{\v{r}}{\'{i}}rodov{\v{e}}d.
  spol. v Brn{\v{e}} (Acta Societ. Scient. Natur. Moravicae) \textbf{3} (1926),
  37--58 (Russian).

\bibitem{Ha}
F.~Buckley and F.~Harary, \emph{Distance in graphs}, Addison-Wesley Publishing
  Company, Advanced Book Program, Redwood City, CA, 1990. \MR{1045632
  (90m:05002)}

\bibitem{B}
P.~Buneman, \emph{A note on the metric properties of trees}, J. Combinatorial
  Theory Ser. B \textbf{17} (1974), 48--50. \MR{0363963 (51 \#218)}

\bibitem{DD}
M.~M. Deza and E.~Deza, \emph{Encyclopedia of distances}, third ed., Springer,
  Heidelberg, 2014. \MR{3243690}

\bibitem{G}
M.~Gromov, \emph{Hyperbolic groups}, Essays in group theory, Math. Sci. Res.
  Inst. Publ., vol.~8, Springer, New York, 1987, pp.~75--263. \MR{919829
  (89e:20070)}

\bibitem{HEF}
M.~Hayamizu, H.~Endo, and K.~Fukumizu, \emph{A characterization of minimum
  spanning tree-like metric spaces}, arXiv:1510.09155 [q-bio.QM] (2015),
  preprint.

\bibitem{H}
M.~D. Hendy, \emph{The path sets of weighted partially labelled trees},
  Australas. J. Combin. \textbf{5} (1992), 277--284. \MR{1165808 (93j:05152)}

\bibitem{I}
V.~Imrih and {\`{E}}.~Stocki{\u{i}}, \emph{The optimal embeddings of metrics
  into graphs}, Sibirsk. Mat. {\v{Z}}. \textbf{13} (1972), 558--565 (Russian).
  \MR{0297621 (45 \#6675)}

\bibitem{K}
J.~B. Kruskal, Jr., \emph{On the shortest spanning subtree of a graph and the
  traveling salesman problem}, Proc. Amer. Math. Soc. \textbf{7} (1956),
  48--50. \MR{0078686 (17,1231d)}

\bibitem{Q}
P.~Qiu, E.~F. Simonds, S.~C. Bendall, K.~D. Gibbs, Jr., R.~V. Bruggner, M.~D.
  Linderman, K.~Sachs, G.~P. Nolan, and S.~K. Plevritis, \emph{Extracting a
  cellular hierarchy from high-dimensional cytometry data with spade}, Nature
  biotechnology \textbf{29} (2011), no.~10, 886--891.

\bibitem{SS}
C.~Semple and M.~Steel, \emph{Phylogenetics}, Oxford Lecture Series in
  Mathematics and its Applications, vol.~24, Oxford University Press, Oxford,
  2003. \MR{2060009 (2005g:92024)}

\bibitem{P1}
J.~M.~S. Sim{\~o}es-Pereira, \emph{A note on the tree realizability of a
  distance matrix}, J. Combinatorial Theory \textbf{6} (1969), 303--310.
  \MR{0237362 (38 \#5650)}

\bibitem{P2}
\bysame, \emph{An optimality criterion for graph embeddings of metrics}, SIAM
  J. Discrete Math. \textbf{1} (1988), no.~2, 223--229. \MR{941352 (90c:05077)}

\bibitem{S}
{\`{E}}.~D. Stocki{\u{i}}, \emph{On the imbedding of finite metrics into a
  graph}, Sibirsk. Mat. {\v{Z}}. \textbf{5} (1964), 1203--1206 (Russian).
  \MR{0171272 (30 \#1503)}

\end{thebibliography}

\end{document}